\documentclass[11pt]{article}
\usepackage{graphicx}
\usepackage{amsmath,amsthm,amssymb}
\usepackage{xypic}
\usepackage{mathabx}
\usepackage{bbm}
\usepackage{color}
\newtheorem{theorem}{Theorem}[section]
\newtheorem{proposition}[theorem]{Proposition}
\newtheorem{definition}[theorem]{Definition}
\newtheorem{example}[theorem]{Example}

\newtheorem{remark}[theorem]{Remark}
\newtheorem{corollary}[theorem]{Corollary}
\newtheorem{lemma}[theorem]{Lemma}
\newtheorem{warning}[theorem]{Warning}

\newtheorem{question}[theorem]{Question}

\newcommand\blfootnote[1]{
  \begingroup
  \renewcommand\thefootnote{}\footnote{#1}
  \addtocounter{footnote}{-1}
  \endgroup
}

\begin{document}

\title{Euler characteristic and homotopy cardinality}
\author{John D. Berman}
\maketitle

\begin{abstract}
Baez asks whether the Euler characteristic (defined for spaces with finite homology) can be reconciled with the homotopy cardinality (defined for spaces with finite homotopy). We consider the smallest $\infty$-category $\text{Top}^\text{rx}$ containing both these classes of spaces and closed under homotopy pushout squares. In our main result, we compute the K-theory $K_0(\text{Top}^\text{rx})$, which is freely generated by equivalence classes of connected $p$-finite spaces, as $p$ ranges over all primes. This provides a negative answer to Baez's question globally, but a positive answer when we restrict attention to a prime.
\end{abstract}

\section{Introduction}
\subsection{The problem}
\noindent A fundamental problem in algebraic topology is the assignment of numerical invariants to (weak equivalence classes of) topological spaces. The oldest and most basic of these is the Euler characteristic, which is uniquely defined on finite CW complexes by the following three properties:\blfootnote{The author was supported by an NSF Postdoctoral Fellowship under grant 1803089.}
\begin{enumerate}
\item $\chi(\emptyset)=0$;
\item $\chi(\ast)=1$;
\item if $D\cong B\,\cup_A^hC$ is a homotopy pushout, $\chi(A)+\chi(D)=\chi(B)+\chi(C)$.
\end{enumerate}

\noindent In particular, the Euler characteristic is invariant under weak equivalence and turns disjoint unions into sums.

In addition to the additive property (3), the Euler characteristic satisfies a multiplicative property:
\begin{enumerate}\setcounter{enumi}{3}
\item if $F\rightarrow E\rightarrow B$ is a fiber sequence of finite CW complexes and $B$ is connected, then $\chi(E)=\chi(F)\chi(B)$.
\end{enumerate}

\noindent In the spirit of Euler, it is tempting to use (4) to `compute' Euler characteristics of more general spaces. If $G$ is a finite group, there is a fiber sequence $G\rightarrow\ast\rightarrow BG$, whence we might deduce $\chi(BG)=\frac{1}{|G|}$.

This is an old idea, dating at least to a 1961 note of Wall \cite{Wall}, where he uses such informal Euler characteristic computations to calculate (for example) that any index 2 subgroup of $\mathbb{Z}/2\ast\mathbb{Z}/3$ is isomorphic to $\mathbb{Z}/3\ast\mathbb{Z}/3$.

\begin{remark}
\label{EarlyRmk}
More generally, suppose $X$ is connected, has finite homotopy groups, and its homotopy groups above a certain degree are zero. Then we may form a Postnikov tower and iteratively apply (4) to conclude that $$\chi(X)=\frac{|\pi_0X||\pi_2X|\cdots}{|\pi_1X||\pi_3X|\cdots}.$$ We will refer to such a space $X$ as \emph{$\pi_\ast$-finite}.

This alternating product is sometimes called the \emph{homotopy cardinality} of $X$. It was studied by Baez and Dolan \cite{BaezDolan} due to its connections with mathematical physics and also plays a role in probabilistic group theory \cite{Tao}.
\end{remark}

\begin{question}[Baez \cite{Baez}]
Is there a well-defined invariant satisfying reasonable properties, which agrees with the Euler characteristic on finite CW complexes and agrees with the homotopy cardinality on $\pi_\ast$-finite spaces?
\end{question}

\subsection{Results}
\noindent The answer is no, if we expect the invariant to behave well with homotopy pushouts. Indeed, if $p\neq q$ are prime, then there is a homotopy pushout square $$\xymatrix{
B\mathbb{Z}/pq\ar[r]\ar[d] &B\mathbb{Z}/p\ar[d] \\
B\mathbb{Z}/q\ar[r] &\ast.
}$$ If there were such an invariant $\chi$, we would have $$\chi(B\mathbb{Z}/pq)+\chi(\ast)=\chi(B\mathbb{Z}/p)+\chi(\mathbb{Z}/q),$$ which is to say $\frac{1}{pq}+1=\frac{1}{p}+\frac{1}{q}$ for all primes $p\neq q$. Of course this is nonsense.\footnote{This neat argument was pointed out to the author by Schlank and Yanovski \cite{SY}.}

However, if we restrict our attention only to $\pi_\ast$-finite spaces for which all homotopy groups are $p$-groups, then there \emph{is} such an invariant. In fact, we don't even have to restrict attention to a single prime $p$. We will prove the following stronger result (following some notation):

\begin{definition}
A space is \emph{$p$-finite} if it is $\pi_\ast$-finite and all homotopy groups are $p$-groups. (Warning: in this paper, any space which is $\pi_\ast$-finite or $p$-finite is assumed to be connected. This differs from many authors.)

A space is \emph{of rational characteristic} if it can be built from $\pi_\ast$-finite spaces out of finite homotopy colimits (that is, out of iterated disjoint unions and homotopy pushouts).
\end{definition}

\begin{theorem}
\label{Thm1}
If $f$ is any function from equivalence classes of $p$-finite spaces to an abelian group $A$ (where $p$ varies over all primes), then there is a unique extension of $f$ to rational characteristic spaces subject to the properties:
\begin{enumerate}
\item $f(\emptyset)=0$;\setcounter{enumi}{2}
\item if $D\cong B\,\cup_A^hC$ is a homotopy pushout, $\chi(A)+\chi(D)=\chi(B)+\chi(C)$.
\end{enumerate}
\end{theorem}

\noindent Let $\text{Top}^\text{rx}$ denote the $\infty$-category of rational characteristic spaces, and $\text{Top}^\text{rx}_p$ a $p$-complete analogue that we define in Section 2. Then Theorem \ref{Thm1} amounts to the assertion that the algebraic K-theory $K_0(\text{Top}^\text{rx})$ is freely generated by the $p$-finite spaces.

There are two ideas here: first, that the $p$-finite spaces form a kind of basis for $\text{Top}^\text{rx}_p$ (linearly independent in the sense of finite colimits); second, that $\text{Top}^\text{rx}$ splits K-theoretically into its prime pieces. In summary we have the following reformulation of Theorem \ref{Thm1}:

\begin{corollary}
\label{Cor1}
There is an equivalence $$K_0(\text{Top}^\text{rx})\cong\bigoplus_p\bigoplus_{\text{Top}_p^\text{fin}}\mathbb{Z},$$ where the direct sum is taken over all equivalence classes of $p$-finite spaces.
\end{corollary}

\noindent In particular, there is a (unique) invariant $\chi$ which
\begin{itemize}
\item agrees with the Euler characteristic on finite CW complexes;
\item agrees with the homotopy cardinality on $p$-finite spaces (for all $p$);
\item satisfies the pushout property (3).
\end{itemize}

\noindent This partially answers Baez's question. However, $\chi$ also:
\begin{itemize}
\item does \emph{not} agree with the homotopy cardinality on $\pi_\ast$-finite spaces with torsion of two different primes (for example, $\chi(B\mathbb{Z}/pq)=\frac{1}{p}+\frac{1}{q}-1$);
\item does \emph{not} satisfy the fibration property (4).
\end{itemize}

\noindent Everything that follows is a proof of Theorem \ref{Thm1}. The proof is remarkable for being very explicit. In Section 2, we restrict attention to one prime and use the Sullivan Conjecture (also known as Miller's Theorem) to construct a rich family of `$p$-complete Euler characteristics' $\chi_K$, one for each $p$-finite space $K$ (Proposition \ref{Prop}). They have the property that $\chi_K(K^\prime)$ measures the number of homotopy classes of maps $K\rightarrow K^\prime$. We also state the $p$-complete analogues of our main results: these are Theorem \ref{Thm2} and Corollary \ref{Cor2}.

Section 3 contains the main proof. In the heart of the proof, we use the Euler characteristics defined in Section 2 to construct (for any $p$-finite $K$) an indicator Euler characteristic $\delta_K$ satisfying $\delta_K(K)>0$ and $\delta_K(K^\prime)=0$ for any $K^\prime$ which is `no larger' than $K$. The rest of the proof is mostly formal. We construct $\delta_K$ by a special sort of induction on the `size' of the $p$-finite space: First, we consider the $k$-quotients $S$ of $K$ ($k$ varies), in a sense we make precise, starting with $\chi_K$. We iteratively kill off $\chi_K(S)$ for each $k$-quotient $S$ using the principle of inclusion-exclusion (Mobius inversion), then iterate again over all $k$.

This is a general technique for handling $\pi_\ast$-finite spaces iteratively, which the author has found very robust. We hope it will have future applications.

Then in Section 4, we assemble all primes together via another formal argument.

\subsection{Acknowledgment and related work}
\noindent In work which is not yet public, Schlank and Yanovski \cite{SY} have independently constructed (what we call) $p$-complete Euler characteristics via a different method which passes through Morava K-theory. The author is thankful to them for pointing out that our $p$-complete argument in Sections 2 and 3 cannot work globally.

This paper fits into a body of literature on generalized Euler characteristics and finiteness obstructions, in which Mobius inversion and inclusion-exclusion play a frequent role. Notably, Fiore, L\"{u}ck, and Sauer \cite{FLS1} \cite{FLS2} have studied an L2 Euler characteristic which has features of the Baez-Dolan homotopy cardinality and the related Leinster Euler characteristics \cite{Leinster}, and Mobius inversion plays a role in formulas for Euler characteristics of homotopy colimits, as in \cite{Berman} \cite{PS1} \cite{PS2}.

\section{Characteristic functions on $p$-complete spaces}
\subsection{Characteristic functions and $K_0$}
\noindent We begin with a brief discussion of the characteristic functions we will be studying:

\begin{definition}
Suppose $\mathcal{C}$ is an $\infty$-category which admits finite colimits. (Equivalently, it has an initial object $\emptyset$ and homotopy pushouts.) A \emph{characteristic function} on $\mathcal{C}$ is a function $f$ from equivalence classes of objects in $\mathcal{C}$ into some abelian group $A$ such that:
\begin{enumerate}
\item $f(\emptyset)=0$;\setcounter{enumi}{2}
\item if $D\cong B\,\cup_A^hC$ is a homotopy pushout, $f(A)+f(D)=f(B)+f(C)$.
\end{enumerate}
\end{definition}

\noindent For fixed target group $A$, characteristic functions can be added and therefore they form an abelian group $\text{Char}(\mathcal{C},A)$. In fact, there is a universal characteristic function $\chi:\mathcal{C}\rightarrow K_0(\mathcal{C})$, in the sense of a natural isomorphism $$\text{Char}(\mathcal{C},A)\cong\text{Hom}(K_0(\mathcal{C}),A).$$ The abelian group $K_0(\mathcal{C})$ can be described explicitly: it has generators $\chi(X)$ for each $X\in\mathcal{C}$ which is not initial and relations for each homotopy pushout square as in (3).

\begin{remark}
\label{Rmk}
More generally, if $f_i:\mathcal{C}\rightarrow A$ ($i\in I$) is a set of characteristic functions, possibly infinite, such that only finitely many $f_i(X)$ are nonzero for any fixed $X\in\mathcal{C}$, then $\sum_if_i$ is a characteristic function as well.
\end{remark}

\begin{example}
If $\text{Top}^\omega$ denotes finite cell complexes, then $K_0(\text{Top}^\omega)\cong\mathbb{Z}$, and the universal characteristic function is the Euler characteristic (up to a choice of sign).
\end{example}

\subsection{Completion at a prime}
\noindent Now we restrict attention to one prime at a time.

\begin{definition}
A space is \emph{$p$-finite} if it is connected, every homotopy group is a finite $p$-group, and every homotopy group above a certain degree is $0$. (That is, it is a $\pi_\ast$-finite space whose homotopy groups are $p$-groups.) The $\infty$-category thereof is $\text{Top}_p^\text{fin}$.

A \emph{$p$-profinite} space is a Pro-object in $\text{Top}_p^\text{fin}$. The $\infty$-category is $\text{Top}_p^\wedge$.

Finally, a $p$-profinite space is \emph{of rational characteristic} if it is a finite colimit (in $\text{Top}_p^\wedge$) of $p$-finite spaces. The $\infty$-category is $\text{Top}_p^\text{rx}$.
\end{definition}

\noindent In this section and the next, we will prove the following $p$-complete analogues of Theorem \ref{Thm1} and Corollary \ref{Cor1}:

\begin{theorem}
\label{Thm2}
Fix a prime $p$. If $f$ is any function from equivalence classes of $p$-finite spaces to an abelian group $A$, then $f$ extends uniquely to a characteristic function on $\text{Top}_p^\text{rx}$.
\end{theorem}

\begin{corollary}
\label{Cor2}
The $K$-theory of $\text{Top}^\text{rx}_p$ is $$K_0(\text{Top}^\text{rx}_p)\cong\bigoplus_{\text{Top}_p^\text{fin}}\mathbb{Z},$$ where the product is taken over all equivalence classes of finite $p$-spaces.
\end{corollary}

\subsection{The characteristic functions $\chi_K$}
\noindent We have $p$-completed in order to use the following result of Lurie:

\begin{theorem}[Lurie \cite{Lurie}]
\label{ThmLurie}
If $K$ is a $p$-finite space, then the functor $$\text{Map}(K,-):\text{Top}_p^\wedge\rightarrow\text{Top}$$ factors through $\text{Top}_p^\wedge$ and is right exact (preserves finite colimits).
\end{theorem}

\noindent This theorem is really Lurie's version of the Sullivan Conjecture, which extends Miller's original proof \cite{Miller}.

Our proof strategy for Theorem \ref{Thm2} consists of constructing a family $\chi_K$ of characteristic functions on $\text{Top}^\text{rx}_p$, one for each $p$-finite space $K$, then showing this family is so rich that we can build out of them \emph{any} characteristic function we want. First, we will describe this family:

\begin{proposition}
\label{Prop}
For any $p$-finite space $K$, there is a characteristic function $\chi_K:\text{Top}_p^\text{rx}\rightarrow\mathbb{Z}$ such that, for any other $p$-finite space $K^\prime$, $$\chi_K(K^\prime)=|\pi_0\text{Map}(K,K^\prime)|.$$
\end{proposition}

\begin{lemma}
\label{Lem0}
Let $\text{Top}_p^\omega$ denote the $\infty$-category of $p$-profinite finite CW complexes; that is, $p$-profinite spaces built from $\ast$ out of finite colimits. The fully faithful inclusion of $\infty$-categories $\text{Top}_p^\omega\rightarrow\text{Top}_p^\text{rx}$ admits a left adjoint $L$, and for any $p$-finite space $K$, $L(K)$ is contractible.
\end{lemma}

\begin{proof}
We need to show that for each rational characteristic $p$-profinite space $X$, there is a $p$-profinite finite CW complex $L(X)$ such that, for any other $p$-profinite finite CW complex $Y$, any map $X\rightarrow Y$ factors essentially uniquely through $L(X)$. That is, restriction induces a weak equivalence $$\text{Map}(L(X),Y)\rightarrow\text{Map}(X,Y).$$ If $X$ is $p$-finite, then this follows from Theorem \ref{ThmLurie}, with $L(X)\cong\ast$.

If $X\cong\text{colim}_i(K_i)$ is a finite colimit of $p$-finite spaces, $L(X)=\text{colim}_i(\ast)$ satisfies the same property $\text{Map}(L(X),Y)\cong\text{Map}(X,Y)$, since the colimit can be pulled out of both sides.
\end{proof}

\begin{proof}[Proof of Proposition \ref{Prop}]
First, we address the case $K=\ast$ is contractible. Let $L$ denote the left adjoint of Lemma \ref{Lem0}. Since it is left adjoint, it is right exact, and therefore $\chi_\ast=\chi\circ L:\text{Top}_p^\text{rx}\rightarrow\mathbb{Z}$ is a characteristic function, where $\chi$ denotes the ordinary Euler characteristic (through mod $p$ homology). In particular, any $p$-finite $K^\prime$ has $$\chi_\ast(K)=\chi(\ast)=1=|\pi_0\text{Map}(\ast,K^\prime)|,$$ as desired.

By Theorem \ref{ThmLurie}, the functor $\text{Map}(K,-):\text{Top}_p^\text{rx}\rightarrow\text{Top}_p^\text{rx}$ is right exact. Therefore, $\chi_K=\chi_\ast(\text{Map}(K,-))$ is a characteristic function on $\text{Top}_p^\text{rx}$.

If $K^\prime$ is another $p$-finite space, then $\text{Map}(K,K^\prime)$ has homotopy groups in bounded degree (is truncated). Therefore, it is a disjoint union of $p$-finite spaces, so $\chi_\ast$ measures the number of connected components. That is, $$\chi_K(K^\prime)=\chi_\ast(\text{Map}(K,K^\prime))=|\pi_0\text{Map}(K,K^\prime)|,$$ completing the proof.
\end{proof}

\section{Calculation of $p$-complete K-theory}
\noindent In this section we prove Theorem \ref{Thm2}. Our strategy is as follows: we place a total ordering on the set of $p$-finite spaces. Then, for each $p$-finite space $K$, we inductively construct characteristic functions $\delta_K^n$ with the following property (Lemma \ref{Lem2}): if $K^\prime\leq K$ is also $p$-finite, then $\delta_K^n(K^\prime)$ is the number of homotopy classes of maps $K\rightarrow K^\prime$ which are isomorphisms on $\pi_{\geq n}$.

In particular, $\delta_K^\infty=\chi_K$, and $\delta_K^0$ satisfies:
\begin{itemize}
\item $\delta_K^0(K)>0$;
\item $\delta_K^0(K^\prime)=0$ if $K^\prime<K$.
\end{itemize}

\noindent Then we can produce a characteristic function with \emph{any} prescribed values on $p$-finite spaces by taking a linear combination of these $\delta_K^0$ (Remark \ref{Rmk}).

\subsection{Quotients of $p$-finite spaces}
\noindent We begin by describing the ordering on $p$-finite spaces, as well as a related notion of $n$-surjection that we will need for the proof.

\begin{definition}
If $K,K^\prime$ are connected spaces, we say that $K^\prime\lessapprox K$ if there is some $n$ for which:
\begin{itemize}
\item $\pi_i K^\prime\cong\pi_i K$ for all $i>n$;
\item $\pi_n K^\prime\not\cong\pi_n K$;
\item $|\pi_n K^\prime|\leq|\pi_n K|$.
\end{itemize}
\noindent An \emph{$n$-surjection} is a map $K\rightarrow K^\prime$ which is an isomorphism on $\pi_{\neq n}$ and a surjection on $\pi_n$.
\end{definition}

\noindent Note that:
\begin{itemize}
\item if $K\rightarrow K^\prime$ is an $n$-surjection, then $K^\prime\lessapprox K$;
\item if there are $n$-surjections $K\xrightarrow{f} K^\prime\xrightarrow{g} K$, and either $K$ or $K^\prime$ is $p$-finite, then $f,g$ are both weak equivalences.
\end{itemize}

\noindent Therefore, if we fix a $p$-finite space $K$, there is a finite poset $P_K^n$ of (weak equivalence classes of) $n$-surjections $K\rightarrow T$, where $T\geq S$ in the poset means there is a homotopy coherent triangle of $n$-surjections $$\xymatrix{
&K\ar[ld]\ar[rd] &\\
T\ar[rr] &&S.
}$$ 

\begin{warning}
There may be multiple homotopy classes of $n$-surjections $T\rightarrow S$ realizing the inequality $T\geq S$ in $P_K^n$. In other words, restriction along an $n$-surjection $K\rightarrow T$ induces a map $P_T^n\rightarrow P_K^n$ whose image is $\{S\in P_K^n|S\leq T\}$, but which is not necessarily injective.
\end{warning}

\noindent We will later need to apply the principle of inclusion-exclusion (Mobius inversion) to the posets $P_K^n$. However, because of the phenomenon described in the warning above, we will need to apply Mobius inversion in a slightly unconventional way. The following lemma guarantees this is possible.

\begin{lemma}[Mobius inversion]
\label{Lem1}
For each $p$-finite space $K$ and $n\geq 0$, it is possible to choose integers $\mu_K^n(T,S)$ defined for every pair \\$(T\in P_K^n,S\in P_T^n)$, satisfying the defining property: If $f,g:P_K^n\rightarrow\mathbb{Z}$ are two functions, then the following identities are equivalent: $$f(T)=\sum_{S\in P^n_T}g(S^\ast)$$ $$g(T)=\sum_{S\in P^n_T}\mu_K^n(T,S)f(S^\ast).$$ Here $S^\ast$ refers to the image of $S$ under the restriction map $P^n_T\rightarrow P^n_K$.
\end{lemma}

\begin{proof}
If $T,S\in P_K^n$, let $\delta(T,S)$ denote the number of $S^\prime\in P_T^n$ which restrict to $S\in P_K^n$ after restriction along $K\rightarrow T$. Note that $\delta(T,T)=1$ since any $n$-surjection $T\rightarrow T$ is a weak equivalence. Also, $\delta(T,S)=0$ unless $S\leq T$ in $P_K^n$. Then the first identity above is $$f(T)=\sum_{S\leq T}\delta(T,S)g(S).$$ Applying ordinary Mobius inversion for a finite poset, there are constants $\bar{\mu}(T,S)$ for each $S\leq T$ in $P_K^n$, such that $$g(T)=\sum_{S\leq T}\bar{\mu}(T,S)f(S)$$ is equivalent to the first identity.

For any fixed $S\leq T$, let $S_1,\ldots,S_k$ be all the elements of $P_T^n$ which restrict to $S\in P_K^n$ upon restriction along the $n$-surjection $K\rightarrow T$. Define $\mu(T,S_i)$ to be arbitrary integers such that $$\sum_{i=1}^k\mu(T,S_i)=\bar{\mu}(T,S).$$ Then the second identity of the lemma is equivalent to $$g(T)=\sum_{S\leq T}\bar{\mu}(T,S)f(S),$$ so the lemma is proven.
\end{proof}

\subsection{The characteristic functions $\delta_K$}
\begin{lemma}
\label{Lem2}
Suppose that $K$ is $p$-finite and $\pi_{>n}K=0$. Iteratively define $\mathbb{Z}$-valued characteristic functions $\delta_K^i$ ($i\geq 0$) on $\text{Top}^\text{rx}_p$ by $$\delta_K^i=\chi_K,\text{when }i>n$$ $$\delta_K^i=\sum_{S\in P_K^i}\mu_K^i(K,S)\delta_S^{i+1},\text{when }i\leq n.$$ Suppose $K^\prime$ is also $p$-finite and \textbf{either} $K^\prime\lessapprox K$ or $K,K^\prime$ are isomorphic on $\pi_{>i}$. Then $\delta_K^i(K^\prime)$ is the number of homotopy classes of maps $K\rightarrow K^\prime$ which are $(i-1)$-truncated; that is, which are isomorphisms on $\pi_{>i}$ and injective on $\pi_i$.
\end{lemma}

\noindent This is the key technical lemma. It iteratively constructs a characteristic function $\delta_K^0$ with the property: if $K^\prime\lessapprox K$, then $\delta_K^0(K^\prime)$ is nonzero if and only if $K\cong K^\prime$. The intermediary functions $\delta_K^i$ are useful only in building $\delta_K^0$.

\begin{proof}
We will prove the lemma by downward induction on $i$.

When $i>n$, every map $K\rightarrow K^\prime$ is $(i-1)$-truncated because $K,K^\prime$ themselves are $(i-1)$-truncated, so the lemma follows from Proposition \ref{Prop}.

Now assume $i\leq n$. If $K,K^\prime$ do not have isomorphic $\pi_{>i+1}$, then $\delta_K^i(K^\prime)=\sum\mu_K^i(S)\delta_S^{i+1}(K^\prime)=0$, by the induction hypothesis. And indeed, there are no $(i-1)$-truncated maps $K\rightarrow K^\prime$, so the lemma conclusion is satisfied.

Therefore we may assume that $K,K^\prime$ are isomorphic on $\pi_{>i+1}$. Any map $K\xrightarrow{p} K^\prime$ factors $K\xrightarrow{c}S\xrightarrow{t}K^\prime$, where $c$ is $(i-1)$-connected and $t$ is $(i-1)$-truncated (\cite{HTT} 6.5.1), and this factorization is unique up to equivalence (\cite{HTT} 5.2.8.17). That is:
\begin{enumerate}
\item $c$ is an isomorphism on $\pi_{<i}$;
\item $t$ is an isomorphism on $\pi_{>i}$;
\item $\pi_iK\rightarrow\pi_iS\rightarrow\pi_i K^\prime$ factors the composite as a surjection followed by an injection.
\end{enumerate}
\noindent If $p$ itself is $i$-truncated, then it must be an isomorphism on $\pi_{>i}$ (because of the lemma's condition that either $K^\prime\lessapprox K$ or they have the same $\pi_{>i}$). Therefore, if $p$ is $i$-truncated, $c$ must also be an isomorphism on $\pi_{>i}$ and therefore an $i$-surjection.

In summary, any $i$-truncated map $K\xrightarrow{p}K^\prime$ factors (essentially uniquely) as an $i$-surjection followed by an $(i-1)$-truncated map. By the induction hypothesis (which describes $\delta_K^{i+1}$ as $i$-truncated maps), $$\delta_K^{i+1}(K^\prime)=\sum_{S\in P_K^i}\text{Tr}_{i-1}(S,K^\prime),$$ where $\text{Tr}_{i-1}(S,K^\prime)$ denotes the number of homotopy classes of $(i-1)$-truncated maps $S\rightarrow K^\prime$. In fact, for any $i$-surjection $K\rightarrow T$, it is still true that either $K^\prime\lessapprox T$ or $K^\prime,T$ are isomorphic on $\pi_{>i}$, so for the same reason we have $$\delta_T^{i+1}(K^\prime)=\sum_{S\in P_T^i}\text{Tr}_{i-1}(S,K^\prime).$$ Applying Mobius inversion (Lemma \ref{Lem1}) with $T=K$, we find $$\text{Tr}_{i-1}(K,K^\prime)=\sum_{S\in P_K^i}\mu_K^i(K,S)\delta_S^{i+1}(K^\prime),$$ and therefore $\delta^K_i(K^\prime)=\text{Tr}_{i-1}(K,K^\prime)$, completing the proof.
\end{proof}

\begin{proof}[Proof of Theorem \ref{Thm2}]
Suppose $f$ is any function from equivalence classes of $p$-finite spaces to an abelian group $A$. We wish to extend $f$ to a characteristic function on $\text{Top}^\text{rx}_p$.

We will first resolve the case that $A$ is free\footnote{Actually, this argument works just as well when $A$ is torsion-free.}. Note that $\lessapprox$ is a total preorder on equivalence classes of $p$-finite spaces. Extend it (arbitrarily) to a total order $K_0<K_1<\cdots$. By Lemma \ref{Lem2}, $\delta^0_{K_i}(K_j)$ is nonzero if $i=j$ and $0$ if $i>j$. Define coefficients $c_i\in A\otimes\mathbb{Q}$ iteratively by $c_0=\frac{f(K_0)}{\delta^0_{K_0}(K_0)}$ and $$c_n=\frac{f(K_n)}{\delta^0_{K_n}(K_n)}-\sum_{i=0}^{n-1}c_i\frac{\delta^0_{K_i}(K_n)}{\delta^0_{K_n}(K_n)}.$$ By Remark \ref{Rmk}, there is a characteristic function $$\bar{f}=\sum_{i=0}^{\infty}c_i\delta^0_{K_i},$$ taking values in $A\otimes\mathbb{Q}$, and $\bar{f}(K_n)=f(K_n)$ for any $n$ (that is, for any $p$-finite $K_n$).

In fact, $\bar{f}$ takes values in $A\subseteq A\otimes\mathbb{Q}$ (which is a subgroup because $A$ is free). This is because any $X\in\text{Top}^\text{rx}_p$ is built out of $p$-finite spaces by iterated homotopy pushouts, and therefore $\bar{f}(X)$ is an integral linear combination of $\bar{f}(K)=f(K)$ for various $p$-finite spaces $K$. So $\bar{f}$ lifts $f$ as a characteristic function with values in $A$ -- if $A$ is free.

For arbitrary $A$, consider a surjective homomorphism $\phi:\mathbb{Z}^n\rightarrow A$. Let $f^\prime$ be an arbitrary function from equivalence classes of $p$-finite spaces to $\mathbb{Z}^n$ such that $f=\phi f^\prime$. Lift $f^\prime$ to a characteristic function $\bar{f}^\prime$ on $\text{Top}^\text{rx}_p$, as above, and then define $\bar{f}=\phi\bar{f}^\prime$. As desired, $\bar{f}$ is a characteristic function which lifts $f$.

The only thing that remains to be shown is that there cannot be two distinct extensions of $f$. However, any $p$-profinite space of rational characteristic is built from $p$-finite spaces via finite homotopy colimits, so a characteristic function on $\text{Top}_p^\text{rx}$ is uniquely determined by its values on $p$-finite spaces.
\end{proof}

\begin{proof}[Proof of Corollary \ref{Cor2}]
Theorem \ref{Thm2} asserts a natural isomorphism $$\text{Char}(\text{Top}_p^\text{rx},A)\rightarrow\text{Hom}(\bigoplus_{\text{Top}_p^\text{fin}}\mathbb{Z},A),$$ given by restricting to the $p$-finite spaces. By the Yoneda Lemma, $$K_0(\text{Top}_p^\text{rx})\cong\bigoplus_{\text{Top}_p^\text{fin}}\mathbb{Z},$$ as desired.
\end{proof}

\section{Calculation of K-theory}
\noindent Recall the main result we set out to prove:

\begingroup
\def\thetheorem{\ref{Thm1}}
\begin{theorem}
If $f$ is any function from equivalence classes of $p$-finite spaces to an abelian group $A$ (where $p$ varies over all primes), then $f$ extends uniquely to a characteristic function on $\text{Top}^\text{rx}$.
\end{theorem}
\addtocounter{theorem}{-1}
\endgroup

\begin{proof}
We know there is at least one characteristic function $\text{Top}^\text{rx}\rightarrow\mathbb{Z}$, namely the Euler characteristic in rational homology, $\chi_\mathbb{Q}$. Since any $\pi_\ast$-finite space is rationally contractible, $\chi_\mathbb{Q}(K)=1$ when $K$ is $\pi_\ast$-finite.

Fix a function $f$ as in the theorem statement, and say $f(\ast)=e\in A$. Let $f_p$ denote the restriction of $f$ to $p$-finite spaces for a particular $p$. By Theorem \ref{Thm2}, $f_p$ extends to a characteristic function $\bar{f}_p$ on $\text{Top}_p^\text{rx}$. Since $p$-completion $(-)^\wedge_p:\text{Top}\rightarrow\text{Top}_p^\wedge$ is right exact, there is a characteristic function $f_p$ on $\text{Top}^\text{rx}$ defined by $$f_p(X)=\bar{f}_p(\hat{X}_p).$$ If $K$ is a $q$-finite space, then $f_p(K)=f(K)$ when $q=p$, and $f_p(K)=e$ when $q\neq p$, since the $p$-completion of a $q$-finite space is contractible.

Given a rational characteristic space $X$, it can be written as a finite homotopy colimit of $\pi_\ast$-finite spaces. There are only finitely many primes that appear in any of the torsion of these $\pi_\ast$-finite spaces. Therefore, pick $p$ large enough so that it does not appear in the torsion of any of those spaces. In this case, $f_p(X)=e\chi_\mathbb{Q}(X)$, by induction on the number of cells in the finite colimit.

In particular, by Remark \ref{Rmk}, there is a characteristic function $$f=e\chi_\mathbb{Q}+\sum_p(f_p-e\chi_\mathbb{Q}),$$ and it extends the original function $f$ by construction.

It only remains to be shown that the lift is unique. If there were two lifts of $f$ to characteristic functions on $\text{Top}^\text{rx}$, then their difference would be $0$ on every $p$-finite space, for every $p$. So we have reduced the problem to the following lemma:
\end{proof}

\begin{lemma}
If $f:\text{Top}^\text{rx}\rightarrow A$ is a characteristic function with $f(K)=0$ for every $p$-finite $K$ (and every prime $p$), then $f(X)=0$ identically.
\end{lemma}

\begin{proof}
For any space $X$, say that the \emph{support} of $X$ is the set of primes $p$ such that $\hat{X}_p$ is not contractible. We claim $f(X)=0$ for any $X\in\text{Top}^\text{rx}$ which is rationally contractible and has finite support. In particular, every $\pi_\ast$-finite space satisfies these conditions, and all rational characteristic spaces are built from $\pi_\ast$-finite spaces and $\emptyset$ via iterated homotopy pushouts, so the lemma will follow.

We prove this claim by induction on $n=|\text{supp}(X)|$. If $n=0$, then $X$ is contractible after rationalization or $p$-completion for any prime. Therefore, $X$ itself is contractible and $f(X)=0$.

If $n=1$, then $X\cong\hat{X}_p$ for some $p$. Since $X$ is built via finite colimits from $\pi_\ast$-finite spaces, and $p$-completion preserves colimits, it follows that $\hat{X}_p\cong X$ is built from $p$-finite spaces. Moreover, $f(K)=0$ for all $p$-finite $K$, so $f(X)=0$.

If $n\geq 2$, let $p$ be any prime in the support of $X$, and let $C$ be the cofiber of the completion map $X\rightarrow\hat{X}_p$. Since rationalization and $q$-completion commutes with colimits (including cofibers), $C$ is rationally contractible and $\text{supp}(C)\subseteq\text{supp}(X)-\{p\}$. By the induction hypothesis, $f(C)=0$. Now we have a pushout square $$\xymatrix{
X\ar[r]\ar[d] &\hat{X}_p\ar[d] \\
\ast\ar[r] &C,
}$$ so $f(X)=f(\hat{X}_p)+f(\ast)-f(C),$ which is $0$ by the induction hypothesis, as desired.
\end{proof}

\noindent Corollary \ref{Cor1} follows just as in the proof of Corollary \ref{Cor2}.

\section{Questions}
\noindent Note that Corollaries \ref{Cor1} and \ref{Cor2} together imply that $p$-completion $$\text{Top}^\text{rx}\rightarrow\prod_p\text{Top}_p^\text{rx}$$ becomes an isomorphism in $K_0$.

\begin{question}
What are the $K$-theory spectra of $\text{Top}^\text{rx}$ and $\text{Top}^\text{rx}_p$? Does $K(\text{Top}^\text{rx})$ split in the same way $K_0(\text{Top}^\text{rx})$ does?
\end{question}

\begin{question}
Suppose $\chi$ is the Baez-Dolan Euler characteristic on $\text{Top}^\text{rx}_p$; that is, if $K$ is a $p$-finite space, then $\chi(K)$ is the homotopy cardinality. Is it true that for any fiber sequence $F\rightarrow E\rightarrow B$, $\chi(E)=\chi(F)\chi(B)$?
\end{question}

\end{document}